\documentclass[reqno, 10pt]{amsart}
\usepackage{amsmath, amssymb}

\newtheorem{thm}{Theorem}[section]
 \newtheorem{lem}[thm]{Lemma}
 \newtheorem{prop}[thm]{Proposition}
 \newtheorem{cor}[thm]{Corollary}
\theoremstyle{remark}
\newtheorem{rem}[thm]{Remark}
\begin{document}


\title[Euler-Poincar\'e equations]{On the Euler-Poincar\'e equation with
non-zero dispersion}

\author[D. Li]{Dong Li}
\address[D. Li]
{Department of Mathematics, University of British Columbia, 1984 Mathematics Road,
Vancouver, BC, Canada V6T1Z2 and \\
School of Mathematics, Institute for Advanced Study, 1st Einstein Drive, Princeton NJ, USA 08544}
\email{dli@math.ubc.ca}

\author[X. Yu]{Xinwei Yu}
\address[X. Yu, Z. Zhai]{Department of mathematical and Statistical Sciences, University of Alberta, Edmonton, AB, T6G 2G1, Canada}
 \email{ xinweiyu@math.ualberta.ca, zhichun1@ualberta.ca}
\author[Z. Zhai]{Zhichun Zhai}



\begin{abstract}
We consider the Euler-Poincar\'e equation on $\mathbb R^d$, $d\ge 2$.
For a large class of smooth initial data we prove that the
corresponding solution blows up in finite time. This settles an open
problem raised by Chae and Liu \cite{Chae Liu}. Our analysis
exhibits some new concentration mechanism and hidden monotonicity
formula associated with the Euler-Poincar\'e flow. In particular we
show the abundance of blowups emanating from smooth initial data with
certain sign properties. No size restrictions are imposed on the
data. We also showcase a class of initial data for which the
corresponding solution exists globally in time.
\end{abstract}
\maketitle
\section{Introduction}
We consider the following Euler-Poincar\'e equation on $\mathbb R^d$,
$d\ge 2$:
\begin{align} \label{1}
\begin{cases}
\partial_t m + (u\cdot \nabla) m + (\nabla u)^{T} m +(\text{div} u) m=0, \qquad t>0,
\; x \in \mathbb R^d;\\
m=(1-\alpha \Delta) u; \\
u(0,x) = u_0(x), \quad x \in \mathbb R^d.
\end{cases}
\end{align}
Here $u=(u_1,\cdots,u_d):\,\mathbb R^d \to \mathbb R^d$ represents
the velocity and $m=(m_1,\cdots,m_d):\, \mathbb R^d \to \mathbb R^d$
denotes the momentum. The parameter $\alpha>0$ in the second
equation of \eqref{1} corresponds to the square of the length scale.
It is sometimes called the dispersion parameter in the literature.
The notation $(\nabla u)^T$ denotes the transpose of the matrix
$\nabla u$. To avoid any confusion it is useful to recast equation
\eqref{1} in the component-wise form as
\begin{align} \label{1a}
 \partial_t m_i + u_j \partial_j m_i + (\partial_i u_j) m_j+ (\partial_j u_j) m_i =0.
\end{align}
Here and throughout the rest of this paper we shall use the Einstein summation convention. By using the tensor notation,
one can combine the second and the last term in \eqref{1a} and write it more compactly as
\begin{align} \label{1b}
 \partial_t m + \nabla \cdot ( m \otimes u) + (\nabla u)^{T} m=0.
\end{align}
The last term in \eqref{1b} is not in conservative form. Following Chae and Liu \cite{Chae Liu} (see formula (1)--(4) on page 673 therein),
one can introduce a stress-tensor $T_{ij}$
\begin{align*}
 T_{ij} = m_i u_j + \frac 12 \delta_{ij} |u|^2 - \alpha \partial_i u \cdot \partial_j u + \frac 12 \alpha \delta_{ij} |\nabla u|^2
\end{align*}
and rewrite \eqref{1b} as
\begin{align} \label{1c}
 \partial_t m_i +  \partial_j T_{ij}=0.
\end{align}
By the second equation in \eqref{1}, we have
\begin{align*}
m_i u_j &= u_i u_j - \alpha (\partial_{kk} u_i) u_j \notag \\
& = u_i u_j - \alpha \partial_k( (\partial_k u_i) u_j) + \alpha (
\partial_k u_i \partial_k u_j).
\end{align*}

Therefore the tensor $T_{ij}$ can be rewritten as
\begin{align}
 T_{ij} &=u_i u_j - \alpha \partial_k( (\partial_k u_i) u_j) + \alpha (
\partial_k u_i \partial_k u_j) \notag \\
&\quad  + \frac 12 \delta_{ij} |u|^2
 - \alpha \partial_i u \cdot \partial_j u + \frac 12 \alpha \delta_{ij} |\nabla u|^2
\label{1d}
\end{align}

Roughly speaking, the above expressions show that the tensor $T$ is of the form
\begin{align*}
 T= O(|u|^2 + |\partial u|^2 + \partial(u \partial u)).
\end{align*}
Such a decomposition is very useful in deriving low frequency $L^p$
estimates later (cf. Proposition \ref{prop1}). For smooth solutions
with enough spatial decay, there are two natural conservation laws
\begin{align}
&\frac d {dt} \int_{\mathbb R^d} m dx =0, \notag \\
&\frac d {dt} \int_{\mathbb R^d} (|u|^2 + \alpha |\nabla u|^2) dx
=0. \label{conserv}
\end{align}
We shall only need the second one for later constructions.

The Euler-Poincar\'e equations were first introduced by Holm, Marsden,
and Ratiu in \cite{Holm Marsden Ratiu1, Holm Marsden Ratiu2}. In 1D
($d=1$) the Euler-Poincar\'e equations reduce to the Camassa-Holm
equations of the form
\begin{align*}
 \partial_t m + u \partial_x m + 2 \partial_x u m=0,\; m=(1-\alpha \partial_{xx})u.
\end{align*}
The well-posedness of local and global weak solutions of Camassa-Holm equations have been intensively studied
(see \cite{Molinet} and references therein).  In 2D, the Euler-Poincar\'e equation is known as the averaged template matching equation in the computer vision literature \cite{Hirani Marsden Arvo,Holm Marsden,  Holm Ratnanather Trouve Younes}. For the applications of
Euler-Poincar\'e equations in computational anatomy, see \cite{Holm Schmah Stoica, Younes}. The rigorous analysis of
the Euler-Poincar\'e equations was initiated by Chae and Liu \cite{Chae Liu}
who established a fairly complete wellposedness theory for
both weak and strong solutions. We summarize some of their main results (relevant to our context)
as follows (here $\alpha$ is the dispersion parameter in the second equation of \eqref{1}):
\begin{itemize}
\item Let $\alpha\ge 0$ and $u_0\in H^k(\mathbb R^d)$ with $k>\frac{d}{2}+3.
$\footnote{For $\alpha=0$ one only needs $k>\frac d2+1$ since the
corresponding system is a symmetric hyperbolic system of
conservation laws with a convex entropy, see Theorem $1$ in
\cite{Chae Liu} for more details.} Then, there exists
$T=T(\|u_0\|_{H^k})>0$ and a unique classical solution $u=u(x,t)$ to
\eqref{1} in the space $ C([0,T), H^k(\mathbb R^d))$.
\item Let $0<T^*\le +\infty$ be the  maximal lifespan corresponding to the solution $u \in C_t^0 H^k$. If $T^*<\infty$, then
\begin{align} \label{2}
\limsup_{t\longrightarrow T^*}\|u(t)\|_{H^k}=\infty\Longleftrightarrow\int_0^{T^*}\|S(t)\|_{\dot{B}^0_{\infty,\infty}}dt=\infty.
\end{align}
Here $S=(S_{ij})$ is the deformation tensor of $u$ with $S_{ij}=\frac{1}{2}(\partial_iu_j+\partial_ju_i).$ See \eqref{besov} for the definition
of the homogeneous Besov norm $\|\cdot \|_{\dot B^0_{\infty,\infty}}$.

\item Let $\alpha=0.$ Let  $u_0\in H^k(\mathbb R^d),$ $k>\frac{d}{2}+2$, and has the reflection symmetry with respect to the origin, i.e.
\begin{align*}
 u_0(x) = - u_0(-x), \qquad \forall\, x \in \mathbb R^d.
\end{align*}
If $\hbox{div}u_0(0)<0$, then the corresponding classical solution blows up in finite time.
\end{itemize}

One should notice that the Chae-Liu blowup result stated above is only valid for $\alpha=0$ in which case the Euler-Poincar\'e equation
reduces to a version of high-dimensional Burgers system. The main idea of Chae-Liu is to consider the evolution of $\hbox{div} u$ at the origin.
Namely by using the reflection symmetry and \eqref{1a}, one obtains
\begin{align*}
 \frac d {dt} (\hbox{div} u(0,t) )& = -2 \Bigl( \sum_{i,j=1}^d (\partial_i u_j + \partial_j u_i) \Bigr)^2 - (\hbox{div} u(0,t))^2 \\
 & \le - (\hbox{div} u(0,t))^2,
\end{align*}
and blowup follows from the assumption $\hbox{div} u_0(0)<0$.
Unfortunately for the non-degenerate case $\alpha>0$, this elegant
argument does not work anymore due to some extra high order terms
which do not enjoy any monotonicity property. Thus Chae and Liu
raised the following

\textbf{Problem: } \emph{for the Euler-Poincar\'e system \eqref{1} ($\alpha>0$),
do there exist finite time blowups from smooth initial data?}

The main purpose of this paper is to settle the above problem in the
affirmative. Since we are mainly interested in the case $\alpha>0$,
the actual value of $\alpha$ will play no role in our analysis.
Henceforth we shall set $\alpha=1$ throughout the rest of this
paper. We start by considering a special class of \emph{radial}
flows invariant under the Euler-Poincar\'e dynamics. More precisely
let $m=\nabla \phi$ where $\phi$ is a radial scalar-valued
function.\footnote{By using the derivation below, it is not
difficult to check that if initially $m_0=\nabla \phi_0$ and
$\phi_0$ is a smooth radial function, then for any $t>0$ we can
write $m(t)=\nabla \phi(t)$ with $\phi(t)$ being radial and smooth
as well. The radial assumption here is essential. In the general
case one can not expect that irrotational flows are preserved in
time.} By \eqref{1a} and noting that $\partial_j m_i =\partial_i
m_j$ for any $i$, $j$, we have
\begin{align} \label{100}
- \partial_t m_i  & = m_i \partial_j u_j + u_j \partial_j m_i + \partial_i u_j m_j \notag \\
&= m_i \partial_j u_j + u_j \partial_i m_j+ \partial_i u_j m_j  \notag \\
&= m_i (\nabla \cdot u) + \partial_i (m \cdot u).
\end{align}
Therefore the radial function $\phi$ satisfies
\begin{align} \label{102}
- \partial_t \phi^{\prime}(r,t)
& = - \phi(r,t) \phi^{\prime}(r,t) +  \phi^{\prime}(r,t) ((1- \Delta)^{-1} \phi)(r,t) \nonumber\\
&+
 \Bigl( (1- \Delta)^{-1} \nabla  \phi \cdot \nabla \phi \Bigr)^{\prime},\quad r=|x|>0,
\end{align}
with initial data $\phi(r,0)=\phi_0(r).$ Here and throughout the rest
of this paper, we will slightly abuse the notation and denote any
radial function $f$ on $\mathbb R^d$ as $f(x)=f(|x|) =f(r)$ whenever
there is no confusion. We also use the notation
$f^{\prime}=f^{\prime}(r)$ to denote the radial derivative. Assuming
$\phi$ (and its derivatives) decays sufficiently fast at infinity,
we may integrate \eqref{102} on the slab $[r,\infty)$ and obtain
\begin{align}
\partial_t \phi(r,t) &= \frac 12 \phi(r,t)^2 + \int_r^{\infty}
\phi^{\prime}(s,t) ( (1-\Delta)^{-1} \phi)(s,t) ds \notag \\
&\qquad -\Bigl( (1-\Delta)^{-1} \nabla  \phi \cdot \nabla \phi
\Bigr)(r,t). \label{102a}
\end{align}
At the cost of a nonlocal integration, the equation \eqref{102a}
simplifies greatly the analysis and will be our main object of study
in this paper. We begin with a simple proposition which in some
sense justifies the validity of the equation \eqref{102a}.

\begin{prop} \label{prop1}
Let the dimension $d\ge 2$. Assume initially $m_0= \nabla \phi_0$
where $\phi_0$ is a radial function on $\mathbb R^d$ and $\phi_0 \in
H^k$ for some $k>\frac d2+4$. If $d=2$, we also assume $\phi_0 \in
\dot B^0_{1,\infty}(\mathbb R^2)$. Then for any $t>0$ the solution
$m(t)=(1-\Delta) u(t)$ can be written as $m(t)=\nabla \phi(t)$ where
$\phi(t)$ is radial and $\phi(t) \in H^k(\mathbb R^d)$ for $d\ge 3$,
$\phi(t) \in H^k(\mathbb R^2) \cap \dot B^0_{1,\infty}(\mathbb R^2)$
for $d=2$. Each $\phi(t,r)$ solves \eqref{102a} in the classical
sense. Moreover we have the growth estimate
\begin{align}
\|\phi(t,\cdot)\|_{L^2(\mathbb R^d)} & \le B_1 \cdot
(1+t)^{\frac12}, \;
\forall\ t\ge 0, \; \text{if $d=2$}, \label{hk_1} \\
\| \phi(t,\cdot) \|_{L^2(\mathbb R^d)} & \le B_2 \cdot (1+t),\;
\forall\, t\ge 0, \; \text{if $d\ge 3$}. \label{hk_2}
\end{align}
Here $B_1>0$, $B_2>0$ are some constants depending only on the
initial data $\phi_0$.
\end{prop}

With Proposition \ref{prop1}  in hand, we can control the low frequency part of the solution and
express the blowup/continuation
\eqref{2} in terms of the scalar function $\phi$ alone. Thus

\begin{lem} \label{lem2new}
Let $\phi_0$ be radial. If $d\ge 3$, we assume $\phi_0 \in H^k(\mathbb R^d)$ for some $k>\frac d2+4$.
If $d=2$, we assume $\phi_0 \in H^k(\mathbb R^2) \cap \dot B^0_{1,\infty}(\mathbb R^2)$ for some $k \ge 6$.
Let $u$ be the maximal lifespan solution corresponding to initial data $u_0 = (1-\Delta)^{-1} m_0
= (1-\Delta)^{-1} \nabla \phi_0$. If the maximal lifespan $T^*<\infty$, then

\begin{align} \notag
\limsup_{t\longrightarrow T^*}\|u(t)\|_{H^k}=
\infty\Longleftrightarrow\int_0^{T^*}\| \phi(t)\|_{L^{\infty}(\mathbb R^d) }dt=\infty.
\end{align}

\end{lem}

We shall omit the proof of Lemma \ref{lem2new} since it follows directly from \eqref{2},
\eqref{hk_1}--\eqref{hk_2}, and the embedding $L^{\infty}\hookrightarrow \dot B^0_{\infty,\infty}$.

We now state our main results. Apart from regularity assumptions, the first result says that
if initially $\phi_0(0)\ge 0$, then the corresponding solution blows up in finite time. It is a bit
surprising in that such a local condition dictates the whole nonlocal Euler-Poincar\'e dynamics.

\begin{thm} \label{thm_blowup_main}
Let the dimension $d\ge 2$. Let $\phi_0$ be a radial real-valued function on $\mathbb R^d$ such that
$\phi_0 \in H^k(\mathbb R^d)$ for some $k>\frac d2+4$. For $d=2$ we also assume $\phi_0 \in
\dot B^0_{1,\infty}(\mathbb R^2)$. Let the initial velocity $u_0 =(1-\Delta)^{-1} m_0 = (1-\Delta)^{-1} \nabla \phi_0$.
If $\phi_0(0)\ge 0$ and $\phi_0$ is not identically zero, then the corresponding solution blows up
in finite time.
\end{thm}

The next result deals with the opposite scenario $\phi_0(0)<0$. Under the assumption that $\phi_0(r)$ is monotonically
increasing, we show the corresponding solution exists globally in time. In some sense it reveals the nonlinear depletion
 mechanism hidden in the Euler-Poincar\'e dynamics.

\begin{thm}[Global regularity for a class of non-positive monotone data] \label{thm2}
Let the dimension $d\ge 2$. Let $\phi_0$ be a radial real-valued function on $\mathbb R^d$ such that
$\phi_0 \in H^k(\mathbb R^d)$ for some $k>\frac d2+4$. For $d=2$ we also assume $\phi_0 \in
\dot B^0_{1,\infty}(\mathbb R^2)$. If $\phi_0(0)\le 0$ and $\phi_0$ is monotonically increasing on $[0,\infty)$
(i.e. $\phi_0^{\prime}(r)\ge 0$ for any $0\le r<\infty$), then the corresponding solution $u(t)=(1-\Delta)^{-1} \nabla \phi(t)$
exists globally in time. Moreover for any $t>0$, $\phi(t,\cdot)$ is monotonically increasing on $[0,\infty)$.
\end{thm}

We have the following corollary which computes the asymptotics of $\phi(0,t)$ as $t\to \infty$. To allow some
generality we state it as
a conditional result in that we assume the corresponding solution exists globally in time.

\begin{cor}[Asymptotics of $\phi(0,t)$] \label{cor_1}
Let the dimension $d\ge 2$. Let $\phi_0$ be a radial real-valued function on $\mathbb R^d$ such that
$\phi_0 \in H^k(\mathbb R^d)$ for some $k>\frac d2+4$. For $d=2$ we also assume $\phi_0 \in
\dot B^0_{1,\infty}(\mathbb R^2)$. Assume $\phi_0(0)<0$, $u_0=(1-\Delta)^{-1} \nabla \phi_0$ and the corresponding solution $u(t) =(1-\Delta)^{-1} \nabla \phi(t)$ exists globally on $[0,\infty)$. Then
$\phi(0,t)$ is strictly monotonically increasing in $t$ and $\frac d {dt} \phi(0,t)>0$ for any $t\ge 0$.
  There are some constants $C_1>0$, $C_2>0$ such that for $d\ge 3$
\begin{align} \label{est_500}
0< - \phi(0,t) < \frac {C_1} {1+t}, \qquad \forall\, t>0;
\end{align}
and for $d=2$
\begin{align} \label{est_500a}
0<-\phi(0,t)<\frac {C_2}{\log(10+t)}, \qquad\forall\, t>0.
\end{align}
In particular $\lim_{t\to\infty} \phi(0,t)=0$.
\end{cor}

\begin{rem}
The decay rates in \eqref{est_500}--\eqref{est_500a} is probably not optimal. It is an interesting
question to study the long time behavior of global solutions to such systems with no damping or dissipation.
\end{rem}

It is very tempting to conjecture that the single condition $\phi_0(0)<0$ may yield global wellposedness.
Our last result rules out this possibility. We exhibit a family of
 smooth \emph{negative} initial data for which
the corresponding solution blows up in finite time. In particular the initial data $\phi_0$ will
satisfy $\phi_0(0)<0$.
\begin{thm} \label{thm4}
There exists a family $\mathcal A$ of smooth initial data such that the following holds:
\begin{itemize}
\item For each $\phi_0 \in \mathcal A$, we have $\phi_0(x)<0$ for any $x\in \mathbb R^d$.
\item The corresponding solution $u(t)=(1-\Delta)^{-1} \nabla \phi(t)$
blows up in finite time. Moreover $\phi(0,t)$ is a monotonically
increasing function of $\,t$ for each $t$ within the lifespan of the solution.
\end{itemize}
\end{thm}

We conclude the introduction by setting up some

\subsubsection*{Notations}
For any two quantities $X$ and $Y$, we shall write $X\lesssim Y$ if
$X\le CY$ for some harmless constant $C>0$. Similarly we define $X\gtrsim Y$.
We write $X\sim Y$ if both $X\lesssim Y$ and $X \gtrsim Y$ hold.

We will  need to use the
Littlewood-Paley frequency projection operators. Let $\varphi(\xi)$ be a smooth bump
function supported in the ball $|\xi| \leq 2$ and equal to one on
the ball $|\xi| \leq 1$.  For each dyadic number $N \in 2^{\mathbb Z}$ we
define the Littlewood-Paley operators
\begin{align}
\widehat{P_{\leq N}f}(\xi) &:=  \varphi(\xi/N)\hat f (\xi), \notag\\
\widehat{P_{> N}f}(\xi) &:=  [1-\varphi(\xi/N)]\hat f (\xi), \notag\\
\widehat{P_N f}(\xi) &:=  [\varphi(\xi/N) - \varphi (2 \xi /N)] \hat
f (\xi). \label{lp_def}
\end{align}
Similarly we can define $P_{<N}$, $P_{\geq N}$, and $P_{M < \cdot
\leq N} := P_{\leq N} - P_{\leq M}$, whenever $M$ and $N$ are dyadic
numbers.

We recall the following standard Bernstein inequality: for any $1\le p<q\le \infty$,
\begin{align*}
\| P_{N} f \|_{L^q(\mathbb R^d)} \lesssim N^{d(\frac 1p -\frac 1q)} \| P_N f \|_{L^p(\mathbb R^d)}.
\end{align*}
Here $P_{N}$ can be replaced by $P_{<N}$ or $P_{\le N}$.

 For any
$1\le p \le \infty$, the homogeneous Besov norm $\dot
B^0_{p,\infty}$ is defined as
\begin{align} \label{besov}
\|f\|_{\dot B^0_{p,\infty}} = \sup_{{ M\in 2^{\mathbb Z}} } \| P_M
f\|_{L^{p}(\mathbb R^d)}
\end{align}

We need the following interpolation inequality on $\mathbb R^2$:
\begin{align}
\| f\|_{L^2(\mathbb R^2)} \lesssim \| f \|^{\frac 12}_{ \dot B^0_{1,\infty}(\mathbb R^2)}
\cdot \| \nabla f \|_{L^2(\mathbb R^2)}^{\frac 12}. \label{hk_int_1}
\end{align}

The proof of \eqref{hk_int_1} is a standard exercise in Littlewood-Paley calculus. We sketch
it here for the sake of completeness.
\begin{proof}[Proof of \eqref{hk_int_1}]
Let $N_0>0$ be a dyadic number whose value will be chosen later. Then by Bernstein, we have
\begin{align*}
\| f \|_{L^2(\mathbb R^2)}^2 & \lesssim \sum_{N<N_0} N^2 \|P_N f \|_{L^1(\mathbb R^2)}^2
+ \sum_{N\ge N_0} N^{-2} \cdot \| \nabla P_N f \|^2_{L^2(\mathbb R^2)} \notag \\
& \lesssim N_0^2 \| f \|^2_{\dot B^0_{1,\infty}(\mathbb R^2)} + N_0^{-2}
\| \nabla f\|_{L^2(\mathbb R^2)}^2.
\end{align*}
Choosing $N_0 \sim \frac {\| \nabla f \|_{2}}{\| f \|_{\dot B^0_{1,\infty}}}$ then yields the result.
\end{proof}

\subsection*{Acknowledgements}
The research of D. Li is partly supported by NSERC Discovery grant and a
start-up grant from University of British Columbia. The research of X. Yu is partly supported by
NSERC Discovery grant and a start-up grant from University of Alberta.
D. Li was also supported in part by NSF under agreement
No. DMS-1128155. Any opinions, findings and conclusions or recommendations
expressed in this material are those of the authors and do not necessarily reflect
the views of the National Science Foundation.

\section{Proof of Proposition \ref{prop1} and some intermediate
results}

In this section we first give the proof of Proposition \ref{prop1}.
After that we shall deduce several weak blowup results some of which
has certain concentration and/or size restrictions on the initial
data. However the proofs of these results are somewhat simpler and
they serve to illustrate main difficulties in proving the sharp
result Theorem \ref{thm_blowup_main}.

\begin{proof}[Proof of Proposition \ref{prop1}]
Since $m=(1-\Delta)u = \nabla \phi$, we have $u= (1-\Delta)^{-1} \nabla \phi$. By using \eqref{conserv}, we obtain
\begin{align}
\| (1-\Delta)^{-1} \nabla \phi(t)\|_2+ \sum_{i,j=1}^d \| (1-\Delta)^{-1} \partial_i \partial_j \phi(t)\|_2 \lesssim 1, \qquad \forall\, t\ge 0.
\label{phi_conserv}
\end{align}
From \eqref{phi_conserv}, we have
\begin{align*}
\| P_{\ge 1} \phi(t) \|_2 \lesssim 1, \quad\forall\, t\ge 0.
\end{align*}

 By using the local
theory worked out in \cite{Chae Liu}, we then only need to show the
persistence of negative regularity and estimate $\| P_{<1} \phi(t)\|_2$. By \eqref{1c}, we have
\begin{align*}
m_i(t) =m_i(0) - \sum_{j=1}^d \int_0^t (\partial_j T_{ij})(\tau)
d\tau.
\end{align*}
Therefore by \eqref{1d} and Bernstein,
\begin{align*}
\| P_{<1} \phi(t) \|_2 & \lesssim \| P_{<1} \Delta^{-1} \nabla \cdot m(t) \|_2 \notag \\
& \lesssim \|\phi(0)\|_2+ \sum_{j=1}^d \int_0^t \| P_{<1} \Delta^{-1} \partial_i \partial_j T_{ij}(\tau) \|_2 d\tau \notag \\
& \lesssim \|\phi(0)\|_2+ \int_0^t \Bigl( \| u(\tau)\|_2^2 + \| \nabla u(\tau) \|_2^2 \Bigr) d\tau \notag \\
& \lesssim \|\phi(0)\|_2+ C_1 t,\quad \forall\, t\ge 0,
\end{align*}
where $C_1>0$ depends on $\|u_0\|_{H^1}$, and we have used the conservation law \eqref{conserv}. Hence the estimate
\eqref{hk_2} follows.

Similarly by using the fact that
\begin{align*}
\sup_{N \in 2^{\mathbb Z} } \| P_N \Delta^{-1} \partial_i \partial_j
\|_{L^1\to L^1} <\infty,
\end{align*}
we obtain in the case $d=2$,
\begin{align*}
\| P_{<1} \phi(t) \|_{\dot B^0_{1,\infty}(\mathbb R^2)} \le C_2
(1+t), \quad\forall\, t\ge 0,
\end{align*}
where $C_2>0$ depends only on $\phi_0$. The growth estimate
\eqref{hk_1} then follows from the above estimate, the conservation
law $\| \nabla P_{<1} \phi\|_2 \lesssim \| P_{<1} (1-\Delta)u \|_2
\lesssim \|u\|_2 \lesssim 1$, and the interpolation inequality
\eqref{hk_int_1} (applied to $f=P_{<1} \phi$).

Finally we need to justify \eqref{102a}. In particular we need to show that the integral
$\int_r^{\infty} (\phi^{\prime})(s,t) ((1-\Delta)^{-1} \phi)(s,t) ds$ converges. Indeed this follows
from the estimate
\begin{align*}
\int_r^{\infty} |\phi^{\prime}| \cdot |(1-\Delta)^{-1} \phi| ds & \lesssim
\left\| \frac{ |\nabla \phi| \cdot (1-\Delta)^{-1} \phi} {|x|^{d-1}} \right\|_{L^1(\mathbb R^d)} \notag \\
& \lesssim \| \nabla \phi\|_{\infty} \cdot \| (1-\Delta)^{-1} \phi \|_{\infty}
+ \| \nabla \phi\|_2 \cdot \| (1-\Delta)^{-1} \phi\|_2 \notag \\
& <\infty.
\end{align*}
Since $\phi \in H^k$, $\phi$ is a smooth function. Since the above integral converges, it follows
that \eqref{102a} holds in the classical sense.
\end{proof}

We now formulate a simple blowup result which requires three rather restrictive conditions on the initial
$\phi_0$: positivity, monotonicity and sufficient concentration at the spatial origin. Due to these
simplifying assumptions, the proof is much simpler
compared to that of our main theorem \ref{thm_blowup_main} in later sections. Note that the case dimension
$d=1$ is covered here which cannot be handled by Theorem \ref{thm_blowup_main}.

\begin{thm} \label{thm1a}
Let the dimension $d\ge 1$. Assume $\phi_0$ is a radial\footnote{In 1D, we simply
require $\phi_0$ is an even function.} real-valued function  on ${\mathbb R}^d$ and $\phi_0 \in H^k(\mathbb R^d)$ for some $k>\frac d2+4$. Assume
$\phi_0^{\prime}(r)\le 0$ for any $r>0$ and $\phi_0$ is not identically zero. There exists a constant $C>0$ such that
if
\begin{align*}
\phi_0(0) \ge C \| \phi_0 \|_{L_x^2(\mathbb R^d)}
\end{align*}
and  the initial velocity $u_0 = (1-\Delta)^{-1} m_0 = (1-\Delta)^{-1} \nabla \phi_0$, then the corresponding solution
 blows up in finite time.
 \end{thm}

\begin{proof}[Proof of Theorem \ref{thm1a}]
Note that by assumption we have $\phi_0$ attains its global maximum at $r=0$ and $\phi_0(0)>0$. We first show that
for any $t>0$ within the lifespan of the solution, we have $\phi^{\prime}(r,t) \le 0$ for any $r>0$. Indeed by
\eqref{102}, we have
\begin{align*}
-\partial_t \phi^{\prime}(r,t) & = - \phi(r,t) \phi^{\prime}(r,t) +\phi^{\prime}(r,t)
( (1-\Delta)^{-1} \phi)(r,t) \notag \\
& \qquad + \Bigl(  (( 1-\Delta)^{-1} \phi)^{\prime} \cdot \phi^{\prime} \Bigr)^{\prime}.
\end{align*}

Set $g(r,t)=\phi^{\prime}(r,t)$, then by using the above equation and grouping the coefficients, we see that
\begin{align*}
\partial_t g(r,t) + a_1(r,t) g(r,t) + a_2(r,t) \partial_r g(r,t) =0, \quad \forall r\ge 0,
\end{align*}
where $a_1$, $a_2$ are some smooth functions. Since $g(r,0)=\phi_0^{\prime}(r) \le 0$,
a simple method of characteristics argument then yields immediately that $g(r,t)\le 0$, $\forall\, r\ge0$.
Hence $\phi^{\prime}(r,t)\le 0$, for any $r\ge 0$.

Now set $r=0$ in \eqref{102a}, we obtain
\begin{align}
\frac d {dt} \phi(0,t)= \frac 1 2 \phi(0,t)^2+ \int_0^{\infty} \phi^{\prime}(s,t)
((1-\Delta)^{-1} \phi)(s,t)ds. \label{106}
\end{align}

By using an argument similar to the derivation of \eqref{hk_1}--\eqref{hk_2} (here we are treating all dimensions $d\ge 1$), we have for all $t\ge 0$,
\begin{align}
\| \phi(t) \|_{L^2(\mathbb R^d)} & \lesssim \| P_{<1} \phi(t) \|_{L^2(\mathbb R^d)} + \| P_{\ge 1}
 \phi(t) \|_{L^2(\mathbb R^d)} \notag \\
& \lesssim \| \phi_0 \|_{L^2(\mathbb R^d)} + \| u_0 \|_{ H^1(\mathbb R^d)} t \notag \\
& \lesssim \| \phi_0\|_{L^2(\mathbb R^d)} (1+t), \label{hk_thm1a_2}
\end{align}
where we have used the relation $u_0=(1-\Delta)^{-1} \nabla \phi_0$.

Since $\phi^{\prime}(r,t)\le 0$ for any $r\ge 0$, we have $\| \phi(t)\|_{\infty} = \phi(0,t)$. Therefore by
\eqref{hk_thm1a_2}, we have
\begin{align}
\| (1-\Delta)^{-1} \phi\|_{\infty}
& \le C \| \phi(t)\|_2+ \frac 1 {100} \|\phi(t)\|_{\infty} \notag \\
& \le C \|\phi_0\|_2 + \frac 1 {100} \phi(t,0). \label{110}
\end{align}

Plugging \eqref{110} into \eqref{106} and using the fact that $(1-\Delta)^{-1} \phi \ge 0$,
$\phi^{\prime} \le 0$, we have
\begin{align}
\frac d {dt} \phi(0,t) \ge \frac 14 \phi(0,t)^2 -  C\| \phi_0 \|_2 (1+t) \cdot \phi(0,t).
\label{hk_thm1a_3}
\end{align}

Clearly for $\phi_0(0)>0$ sufficiently large (compared to $\| \phi_0\|_2$), $\phi(0,t)$ will blow up in finite time.
\end{proof}

Our next result refines Theorem \ref{thm1a} in that it removes the size assumption on $\phi_0$. For some
technical reasons (see Lemma \ref{lem100a}), it only treats
 dimensions $d\ge 3$.

\begin{thm} \label{thm1b}
Let the dimension $d\ge 3$. Assume $\phi_0$ is a radial real-valued function on $\mathbb R^d$
 and $\phi_0 \in H^k(\mathbb R^d)$ for some $k>\frac d2+4$. Assume
$\phi_0^{\prime}(r)\le 0$ for any $r>0$ and $\phi_0$ is not identically zero.
If  the initial velocity $u_0 = (1-\Delta)^{-1} m_0 = (1-\Delta)^{-1} \nabla \phi_0$,
then the corresponding solution
 blows up in finite time.
 \end{thm}

The proof of Theorem \ref{thm1b} relies on the following lemma which can be regarded as some type of
Poincar\'e inequality.

\begin{lem} \label{lem100}
Let the dimension $d\ge 1$.
For any $C_1>0$, $1\le p <\infty$, there is a constant $\epsilon_0>0$ depending only
on $C_1$, $p$ and the dimension $d$ such that the following holds:

Suppose $f:\mathbb R^d\to \mathbb R$ is a function (not necessarily radial) such that
\begin{align*}
0\le f(x) \le f(0), \qquad\forall \, x\in \mathbb R^d.
\end{align*}

 If $\| f\|_{L_x^p}
\le C_1 |f(0) |<\infty$, then

\begin{align}
\Bigl| (\frac {\Delta} {1-\Delta} f )(0) \Bigr| \ge \epsilon_0 |f(0)|. \label{500a}
\end{align}

\end{lem}

\begin{proof}[Proof of Lemma \ref{lem100}]
WLOG we may assume $f(0)=1$.  Denote the Bessel potential $K(x) = \mathcal F^{-1} ( (1+|\xi|^2)^{-1} ) (x)$.
Note that $K$ is a positive radial function on $\mathbb R^d$ and $K\in L_x^1 \cap L_x^q$ for any $1<q<\frac {d}{d-2}$
(for $d=2$ we have $K\in L_x^1 \cap L_x^{q}$ for any $q<\infty$, and for $d=1$ we have $K\in L_x^1\cap L_x^\infty$).
Then
\begin{align}
\Bigl( \frac {-\Delta}{1-\Delta} f \Bigr)(0) & = f(0) -  ( \frac 1 {1-\Delta} f )(0) \notag \\
& = \int_{\mathbb R^d} K(y) (f(0) - f(y) ) dy. \notag
\end{align}

Assume the bound \eqref{500a} is not true. Then there exists a sequence
of nonnegative  functions
$f_n$ such that  $f_n(0)=1$,  $\| f_n\|_{L_x^{\infty}} \le 1$, $\| f_n \|_{L_x^p} \le C_1$ and
\begin{align}
\int_{\mathbb R^d} K(y) (1- f_n(y) ) dy \to 0. \label{limit100}
\end{align}

Now take a number $r>p$  sufficiently large such that
$K \in L_x^{\frac r {r-1}}$. Obviously $\| f_n \|_{L_x^r} \le C_2<\infty$ for some
constant $C_2>0$ independent of $n$.  By passing to a subsequence in $n$ if necessary, we have $f_n \rightharpoonup g$
weakly in $L_x^r$ for some $g\in L_x^r$. Furthermore we have $\| g\|_{L_x^{\infty}} \le 1$.  By \eqref{limit100} and the fact
$K \in L_x^{\frac r {r-1}}$, we then obtain
\begin{align*}
\int_{\mathbb R^d} K(y) (1- g(y)) dy =0
\end{align*}
which implies $g(y)=1$ for a.e. $y\in \mathbb R^d$. This clearly contradicts to the fact that $g\in L_x^r$. The lemma is proved.
\end{proof}

\begin{rem}
It is also possible to give a constructive proof of Lemma \ref{lem100}.
For example in the 3D case, we have $K(x)=\mathcal F^{-1}((1+|\xi|^2)^{-1})(x)=\frac{1}{4\pi}\frac{e^{-|x|}}{|x|}$.
Let $p^{\prime}=\frac{p}{p-1}$ be the usual H\"older conjugate of $p$. Let $R>0$ be a number whose
value will be chosen later. Then we have
\begin{align}
\int_{\mathbb{R}^3} K(y)f(y) dy & =  \int_{|y|<R} K(y) f(y) dy + \int_{|y|>R} K(y) f(y) dy \notag \\
      & \le \left\| f\right\|_{L^\infty} \int_0^R re^{-r}dr
             +\left(\int_{|y|>R} K^{p'} dy \right)^{1/p'} \left\| f\right\|_{L^p} \notag \\
      & \le \left(1-(R+1)e^{-R}\right)\left\| f\right\|_{L^\infty} \notag \\
         &\qquad      +\left(\frac{1}{(4\pi)^{p'-1}}\int_R^\infty e^{-p'r}r^{2-p'}dr\right)^{1/p'}
                       \left\| f\right\|_{L^p}. \label{rem100_e22}
\end{align}
To estimate \eqref{rem100_e22}, we compute (note that $p'>1$, and
assume that $R>1$)
\begin{align}
\int_R^\infty e^{-p'r}r^{2-p'}dr &\le
                R^{1-p'} e^{(1-p')R} \int_R^\infty e^{-r}r dr \notag \\
        &= R^{1-p'} e^{(1-p')R} (R+1)e^{-R} \notag\\
        &\le C R^{2-p'}e^{-p'R}. \notag
\end{align}
Plugging the above estimate into \eqref{rem100_e22}, we have
\begin{align}
\int_{\mathbb{R}^3} K(y)f(y) dy \leq \left(1-(R+1)e^{-R}\right)\left\| f\right\|_{L^\infty}
                + C R^{(2-p')/p'}e^{-R}\left\| f\right\|_{L^p}. \notag
\end{align}
Since $p'>1$ we have $(2-p')/p'<1$.
Now if there is a constant $C_1$ such that $\left\|f\right\|_{L^p}\leq C_1 \left\|f\right\|_{L^\infty}$
holds, then we can always choose $R$ big enough to obtain
\begin{align}
\int_{\mathbb{R}^3} K(y)f(y) dy \leq (1-\varepsilon_0) \left\|f\right\|_{L^\infty}. \notag
\end{align}
This then leads to \eqref{500a}.
\end{rem}

For the proof of Theorem \ref{thm1b}, we need a slightly modified version of Lemma \ref{lem100}.
Note the dimension restriction $d\ge 3$ and see also Remark \ref{rem_lem100a} below.
\begin{lem} \label{lem100a}
Let the dimension $d\ge 3$.
For any $C_1>0$, $1\le p <\infty$, there is a constant $\epsilon_0>0$ depending only
on $C_1$, $p$ and the dimension $d$ such that the following holds:

Suppose $f:\mathbb R^d\to \mathbb R$ satisfies
\begin{align*}
0\le f(x) \le f(0), \quad\forall\, x \in \mathbb R^d.
\end{align*}

 If $f\in L_x^2(\mathbb R^d)$ and
 $$\Bigl\| \frac{|\nabla|}{1+|\nabla|} f \Bigr\|_{L_x^2}
\le C_1 |f(0) |<\infty,$$ then

\begin{align}
\Bigl| (\frac {\Delta} {1-\Delta} f )(0) \Bigr| \ge \epsilon_0 |f(0)|. \label{lem100a_1}
\end{align}

\end{lem}

\begin{rem} \label{rem_lem100a}
We stress that the dimension restriction $d\ge 3$ is necessary in Lemma \ref{lem100a}.
In dimensions $d=1,2$, there exist counterexamples which are made of approximating sequences
of the constant functions. To see this, let $t>0$ and define
\begin{align*}
f_t(x) = e^{-t |x|^2}, \quad\, x \in \mathbb R^d.
\end{align*}
Then obviously $f_t(0)=1$ and
\begin{align*}
\hat f_t(\xi) = const  \cdot t^{-\frac d2} e^{-\frac{|\xi|^2}{4t}}, \quad \xi \in \mathbb R^d.
\end{align*}
When $d=1,2$, it is not difficult to check that
\begin{align*}
&\Bigl\| \frac {|\nabla|} {1+|\nabla|} f_t(x) \Bigr\|_{L_x^2}^2  \notag \\
\lesssim & \int_{\mathbb R^d} \frac {|\xi|^2} {1+|\xi|^2} \cdot t^{-d} e^{-\frac {|\xi|^2}{2t}} d \xi\\
\lesssim & t^{1-\frac d2} \lesssim 1, \quad \text{as $t\to 0$.}
\end{align*}
Similarly we have
\begin{align*}
 &\Bigl| (\frac {\Delta}{1-\Delta} f_t)(0) \Bigr| \notag \\
\lesssim & \int_{\mathbb R^d} \frac {|\xi|^2}{1+|\xi|^2} \cdot t^{-\frac d2} e^{-\frac{|\xi|^2}{4t}} d\xi \notag \\
\lesssim &\, t^{2} \to 0,\quad \text{as $t\to 0$}.
\end{align*}
Obviously \eqref{lem100a_1} cannot hold in this case.
\end{rem}

\begin{proof}[Proof of Lemma \ref{lem100a}]
Again we will argue by contradiction. Assume \eqref{lem100a_1} does not hold. Then we can find
a sequence of  nonnegative functions $f_n \in L_x^2(\mathbb R^d)$ with $\|f_n\|_{\infty}=f_n(0)=1$ such that
\begin{align}
\Bigl \| \frac {|\nabla|}{1+|\nabla|} f_n \|_{L_x^2} \le C_1, \label{lem100a_2}
\end{align}
and
\begin{align}
\int_{\mathbb R^d} K(y) (1-f_n(y))dy \to 0, \quad \text{as $n\to \infty$}. \label{lem100a_3}
\end{align}
By \eqref{lem100a_2} and passing to a subsequence in $n$ if necessary, we can find $g \in L_x^2(\mathbb R^d)$ such that
\begin{align*}
\frac{|\nabla|}{1+|\nabla|} f_n \rightharpoonup g, \quad \text{weak in $L^2_x(\mathbb R^d)$, as $n\to \infty$}.
\end{align*}

Now for any $\phi \in \mathcal S(\mathbb R^d)$, observe that
\begin{align*}
\frac{1+|\nabla|}{|\nabla|} \phi \in L_x^2(\mathbb R^d), \quad \text{for $d\ge 3$}.
\end{align*}

Therefore
\begin{align}
\int_{\mathbb R^d} f_n \phi dx & = \int_{\mathbb R^d} \frac{|\nabla|}{1+|\nabla|} f_n
\cdot \frac{1+|\nabla|}{|\nabla|} \phi dx \notag \\
& \to \int_{\mathbb R^d} g \cdot \frac{1+|\nabla|}{|\nabla|} \phi dx \notag \\
&=:\; T(\phi), \quad \text{as $n\to \infty.$} \label{you_1}
\end{align}

Since $0\le f_n \le 1$ and
\begin{align*}
T(\phi) = \lim_{n\to \infty} \int_{\mathbb R^d} f_n \phi dx,
\end{align*}
it follows that for $\phi\ge 0$, we have $T(\phi)\ge 0$. Therefore by the Riesz representation
theorem, we have
\begin{align*}
T(\phi) = \int_{\mathbb R^d} \phi d\mu,
\end{align*}
for some non-negative Borel measure $d\mu$.  Now since
\begin{align*}
\left|\int_{\mathbb R^d} f_n \phi dx \right| \le \| \phi \|_{L_x^1(\mathbb R^d)},
\end{align*}
we get
\begin{align*}
\left| \int_{\mathbb R^d} \phi d\mu \right| \le \| \phi \|_{L_x^1(\mathbb R^d)}, \quad \forall\,\phi \in
\mathcal S(\mathbb R^d).
\end{align*}

Therefore in a standard way we can extend $d\mu \in (L_x^1)^*= L_x^{\infty}$. Hence for some $f_{\infty} \in
L_x^{\infty}(\mathbb R^d)$ with $0\le f_n(x) \le 1$, a.e. $x\in \mathbb R^d$, we have
\begin{align*}
T(\phi) = \int_{\mathbb R^d} \phi(x) f_{\infty}(x) dx.
\end{align*}

By a density argument, we obtain
\begin{align*}
\lim_{n\to \infty} \int_{\mathbb R^d} f_n \phi dx = \int_{\mathbb R^d} f_{\infty}  \phi dx,
\quad \forall\, \phi \in L_x^1(\mathbb R^d).
\end{align*}

In particular,
\begin{align*}
\lim_{n \to \infty} \int_{\mathbb R^d} K(x) f_n(x) dx = \int_{\mathbb R^d} K(x) f_{\infty}(x) dx.
\end{align*}

Therefore by \eqref{lem100a_3}
\begin{align*}
\int_{\mathbb R^d} K(x) (1-f_{\infty}(x)) dx =0,
\end{align*}
and obviously $f_{\infty}(x)=1$ for a.e. $x\in \mathbb R^d$.

Plugging this back into \eqref{you_1},  we obtain for any $\phi \in \mathcal S(\mathbb R^d)$,
\begin{align*}
\int_{\mathbb R^d} g \cdot \frac{1+|\nabla|}{|\nabla|} \phi dx= \int_{\mathbb R^d} \phi dx,
\end{align*}
or on the Fourier side,
\begin{align*}
\int_{\mathbb R^d} \hat g(\xi) \cdot \frac{1+|\xi|}{|\xi|} \hat \phi(\xi) d\xi = \hat \phi(0).
\end{align*}

From this and the fact that $\hat g \in L^2$, it follows easily that $\hat g(\xi)=0$ for a.e. $\xi \in \mathbb R^d$.
This is obviously a contradiction.

\end{proof}

We are now ready to complete the

\begin{proof}[Proof of Theorem \ref{thm1b}]
Denote $g= (1-\Delta)^{-1} \phi$. Set $r=0$ in \eqref{102a}, we then rewrite it as
\begin{align}
\frac d{ dt} \phi(0,t) & = \int_0^{\infty} \Bigl( (1-\Delta) g
\Bigr)^{\prime} \Delta g dr \notag \\
& = \int_0^{\infty} g^{\prime} \Delta g dr - \int_0^{\infty} (\Delta
g)^{\prime} \Delta g dr \notag \\
& = \int_0^{\infty} g^{\prime} ( g^{\prime\prime} + \frac {d-1} r
g^{\prime} ) dr  + \frac 1 2 \Bigl ( (\Delta g)(0,t) \Bigr)^2.
\label{thm1b_202}
\end{align}
Now note that $g$ is a radial function, so $g^{\prime}(0,t)=0$. Therefore we have
\begin{align*}
\int_0^{\infty} g^{\prime} g^{\prime\prime} dr =0.
\end{align*}
Therefore we obtain from \eqref{thm1b_202} the following inequality
\begin{align}
\frac d {dt} \phi(0,t) &\ge \frac 12 \Bigl( (\Delta g) (0,t)
\Bigr)^2 \notag \\
& =\frac 12 \Bigl(  \bigl( \frac{\Delta}{1-\Delta} \phi \bigr)(0,t) \Bigr)^2. \label{thm1b_510}
\end{align}

Now by using the energy conservation \eqref{conserv} and the relation $u=(1-\Delta)^{-1} \nabla \phi$, we have
\begin{align*}
\left\| \frac {|\nabla|} {1+|\nabla|} \phi(\cdot, t) \right\|_{L_x^2} \le C_3 <\infty,
\end{align*}
where $C_3$ is some constant independent of $t$.

Note that $\phi(0,t) \ge \phi(0,0) >0$. Since we assume the dimension $d\ge 3$,  by Lemma \ref{lem100a}, we have
\begin{align}
\Bigl|  (\frac{\Delta} {1-\Delta} \phi)(0,t ) \Bigr| \ge \epsilon_0  \phi(0,t), \notag
\end{align}
where $\epsilon_0>0$ is independent of $t$.

Plugging this estimate into \eqref{510}, we obtain
\begin{align*}
\frac d {dt} \phi(0,t) \ge \frac 12 \epsilon_0^2 \phi(0,t)^2
\end{align*}
which clearly implies that $\phi(0,t)$ must blow up in finite time.
\end{proof}

\section{proof of main theorems}

\subsection{Proof of Theorem \ref{thm_blowup_main}}
 Write \eqref{106} as
\begin{align}
\frac {d}{dt} \phi(0,t) &= \int_0^{\infty} \phi^{\prime} \Bigl(
(1-\Delta)^{-1} -1 \Bigr) \phi ds \notag \\
& = \int_0^{\infty} \phi^{\prime} \frac{\Delta}{1-\Delta} \phi ds.
\label{200}
\end{align}
Now denote $g= (1-\Delta)^{-1} \phi$. We then rewrite \eqref{200} as
\begin{align}
\frac {d} { d t} \phi(0,t) & = \int_0^{\infty} \Bigl( (1-\Delta) g
\Bigr)^{\prime} \Delta g dr\notag \\
& = \int_0^{\infty} g^{\prime} \Delta g dr - \int_0^{\infty} (\Delta
g)^{\prime} \Delta g dr \notag \\
& = \int_0^{\infty} g^{\prime} ( g^{\prime\prime} + \frac {d-1} r
g^{\prime} ) dr + \frac 1 2 \Bigl ( (\Delta g)(0,t) \Bigr)^2.
\label{202}
\end{align}
Now note that $g^{\prime}(0,t)=0$ and we have
\begin{align*}
\int_0^{\infty} g^{\prime} g^{\prime\prime} dr =0.
\end{align*}
Therefore we obtain from \eqref{202} the following identity
\begin{align}
\frac d {dt} \phi(0,t) &= (d-1) \int_0^{\infty} \frac {(g^\prime)^2} r dr +  \frac 12 \Bigl( (\Delta g) (0,t)
\Bigr)^2 \notag \\
& = (d-1) \int_0^{\infty} \frac {(g^\prime)^2} r dr +
 \frac 12 \Bigl(  \bigl( \frac{\Delta}{1-\Delta} \phi \bigr)(0,t) \Bigr)^2 \notag \\
& = (d-1) \int_0^{\infty} \frac {(g^{\prime})^2} r dr
 + \frac 12 \Bigl( \phi(0,t) -g (0,t)  \Bigr)^2. \label{510}
\end{align}

Since $\phi_0(0)\ge 0$ and $\phi_0$ is not identically zero, we have that for all $t\ge t_0$,
\begin{align}
\phi(0,t) \ge A_1,  \label{510a}
\end{align}
where  $t_0>0$ is any fixed time and $A_1$ is a constant depending on $\phi_0$ and $t_0$.

Now let $R>1$ be a parameter whose value will be specified later. Note that by the Fundamental Theorem
of Calculus, we have
\begin{align}
|g(0,t) - g(R,t)| & \le \int_0^R | g^\prime | dr \notag \\
& \le \Bigl( \int_0^R \frac {(g^\prime)^2} {r} dr \Bigr)^{\frac 12} \cdot R. \label{540}
\end{align}

Then clearly for dimensions $d\ge 2$,
\begin{align}
\eqref{510} & \ge \frac 1 {100 R^2} \Bigl( \left|\phi(0,t) -g(0,t)\right| + R \bigl( \int_0^R \frac{ (g^\prime)^2} r dr \bigr)^{\frac 12} \Bigr)^2\\
& \ge  \frac{1}{100 R^2} \Bigl( \left|\phi(0,t) -g(0,t)\right| + |g(0,t) - g(R,t)| \Bigr)^2\\
& \ge   \frac{1}{100 R^2}\Bigl( \phi(0,t) -g(R,t) \Bigr)^2
\label{520}
\end{align}

Now we discuss two cases. Consider first the case dimension $d\ge
3$. By  radial Sobolev embedding and energy conservation
\eqref{conserv}, we have
\begin{align}
|g(R,t)| & \le C_d \| \nabla g \|_2 \cdot R^{-\frac{d-2}2} \notag \\
& \le C_d { \| u_0\|_{H^1} } \cdot R^{-\frac{d-2}2}, \label{560}
\end{align}
where $C_d$ is constant depending only on the dimension $d$, and $u_0=(1-\Delta)^{-1}\nabla \phi_0$
is the initial velocity.  By \eqref{560}, we can choose
$R$ sufficiently large such that
\begin{align}
|g(R,t)| \le \frac 1 {100}A_1, \label{570}
\end{align}
where $A_1$ was defined in \eqref{510a}.  Therefore  by \eqref{520}, \eqref{560}, and \eqref{570}, we get for all $t>t_0$,
\begin{align}
    \phi(0,t) - g(R,t) \ge & \frac 12 \phi(0,t). \notag
 \end{align}

Plugging this estimate into \eqref{520}, we obtain for $t>t_0$, and some constant $\epsilon_0>0$,
\begin{align*}
\frac d {dt} \phi(0,t) \ge \frac 12 \epsilon_0 \phi(0,t)^2
\end{align*}
which together with \eqref{510a} clearly implies that $\phi(0,t)$ must blow up in finite time.

This finishes the case $d\ge 3$. Now we turn to the case dimension $d=2$. We shall choose for
each $g(t)$ the time-dependent parameter $R(t)=R_0(1+t)^{\frac 12}$ where $R_0$ will be taken sufficiently large.
By \eqref{hk_1} and radial Sobolev embedding, we have
\begin{align*}
|g(R(t),t)| & \le C \cdot \| \phi(t) \|_2^{\frac 12} \cdot (R(t))^{-\frac 12} \notag \\
& \le C\cdot B_1 \cdot R_0^{-\frac 12}.
\end{align*}

Choosing $R_0$ sufficiently large gives us \eqref{570} and consequently
\begin{align*}
\frac d {dt} \phi(0,t) \ge C \cdot \frac 1 {1+t} \phi(0,t)^2.
\end{align*}

Integrating the above ODE on the interval $[t_0,\tau)$ with $\tau>t_0$, we get
\begin{align*}
- \frac 1 {\phi(0,\tau)} + \frac 1 {\phi(0,t_0)} \ge const \cdot \log (1+\tau).
\end{align*}

This implies that $\frac 1 {\phi(0,\tau)}$ become negative in finite time which obviously
contradicts \eqref{510a}.

\subsection{Proof of Theorem \ref{thm2}}

By repeating an argument similar to the beginning part of the proof of Theorem \ref{thm1a},
we have $\phi^{\prime}(r,t)\ge 0$ for any $r>0$. Set $\phi=-\psi$. Then by \eqref{106},
we have
\begin{align*}
\frac d {dt} \psi(0,t)= - \frac {\psi(0,t)^2}2
- \int_0^{\infty} \psi^{\prime}(r,t)
((1-\Delta)^{-1} \psi)(r,t) dr.
\end{align*}

By a derivation similar to \eqref{hk_thm1a_2}, we then have for any $t>0$,
\begin{align*}
\| \psi(t) \|_{L^2(\mathbb R^d)} = \| \phi(t) \|_{L^2(\mathbb R^d)} \le C\cdot (1+t),
\end{align*}
where $C>0$ depends only on $\phi_0$.

Therefore in place of \eqref{hk_thm1a_3}, we get
\begin{align*}
\frac d {dt} \psi(0,t) \le - \frac{\psi(0,t)^2}4 + C \cdot (1+t) \cdot \psi(0,t).
\end{align*}

Since $\psi(0,0)\ge 0$, this clearly shows that $\psi(0,t)$ is bounded for all $t>0$.
By using the blowup criteria Lemma \ref{lem2new}, we conclude that the corresponding
classical solution exists for all time $t>0$.

\subsection{Proof of Corollary \ref{cor_1}}

The monotonicity of $\phi(0,t)$ follows directly from the proof of Theorem \ref{thm_blowup_main} (see \eqref{510}).
In particular we know
that $\phi(0,t)<0$ for any $t\ge 0$ (otherwise the corresponding solution will blow up). It remains to establish the
estimates \eqref{est_500}--\eqref{est_500a}.
By using the same argument as in the proof of Theorem \ref{thm_blowup_main}, we obtain the inequality
\begin{align*}
\frac d {dt} \phi(0,t) &\ge \epsilon_0 \phi(0,t)^2, \quad \text{if $d\ge 3$}, \\
\frac d {dt} \phi(0,t) & \ge \frac{\epsilon_1}{1+t} \phi(0,t)^2,\quad
\text{if $d=2$,}
\end{align*}
where $\epsilon_0>0$, $\epsilon_1>0$ are some constants.
Integrating the above inequality in time gives us the desired results.

\subsection{Proof of Theorem \ref{thm4}}
Let $\psi_0 \in H^{\infty} (\mathbb R^d) $ be a smooth radial
function such that $\psi_0(0)=0$ and
\begin{align} \label{cond_800}
\begin{cases}
\psi_0^{\prime}(x)\le 0, \quad \, |x| \le c_1,\\
\psi_0^{\prime}(x)> 0, \quad \, |x|>c_2, \\
\psi_0(x)<0, \quad \, c_1/2<|x|<2c_2.
\end{cases}
\end{align}
Here $0<c_1<c_2<\infty$ are arbitrary constants.

By local wellposedness theory, there exists $T_0>0$ and a smooth solution $\psi=\psi(x,t)$ to \eqref{1} ($m=\nabla \psi$)
in the space $C([-T_0,T_0], H^{k})$ for any $k\ge 0$.

We make the following

\textbf{Claim}: there exists $t_0>0$ sufficiently small, such that
$\psi(x,-t_0)<0$ for any $x\in \mathbb R^d$.

We now assume the claim is true and complete the proof of the
theorem. Take $\phi_0(x):= \psi(x,-t_0)$ for $x\in \mathbb R^d$.
Then clearly $\phi_0(x)$ satisfies all the needed conditions.
Furthermore denote the solution corresponding to the data $\phi_0$
as $\phi=\phi(x,t)$. It is obvious that $\phi(x,t)= \psi(x, t-t_0)$
for any $t\ge 0$. In particular we have $\phi(0, t_0)= 0$. By using
 Theorem \ref{thm_blowup_main}, it follows easily that $\phi$ must
blow up at some $t>t_0$. Therefore $\phi_0$ is the desired initial
data.

It remains for us to prove the claim. Write $\psi=\psi(r,t) = \psi(x,t)$, where $r=|x|$.
Note that $\psi \in C^{\infty}([0,\infty))$ as a function of $r$.  We can perform an even extension and regard
$\psi \in C^{\infty}(\mathbb R)$.

By  \eqref{102}, we have
\begin{align}
-\partial_t \psi^{\prime} & = \Bigl( - \psi +  (1-\Delta)^{-1} \psi
+ \left( (1-\Delta)^{-1} \psi \right )^{\prime\prime} \Bigr)\psi^{\prime} + ( (1-\Delta)^{-1} \psi)^{\prime} \cdot \psi^{\prime\prime}
\notag \\
& =: c(r,t) \psi^{\prime} + b(r,t) \cdot \psi^{\prime\prime}. \notag
\end{align}
Here $c=c(r,t)$, $b=b(r,t)$ are both smooth functions for $-T_0\le t \le T_0$ and $ r \in \mathbb R$.  Note that $c$ is an even function
and $b$ is an odd function. Also for some constant $B>0$
\begin{align} \label{805}
\sup_{|t| \le T_0} ( \| b(t,\cdot)\|_{\infty} +\| \partial_r b(t,\cdot) \|_{\infty} ) \le B <\infty.
\end{align}

Denote $f(r,t) =\psi^{\prime}(r,t)$, then $f(r,t)$ satisfies the transport
equation
\begin{align} \label{807}
\partial_t f + b \partial_r f + c f =0.
\end{align}

Introduce the characteristic lines:
\begin{align*}
\begin{cases}
\frac d{ dt} z(t, \alpha) = b(z(t,\alpha), t), \\
z(0,\alpha) = \alpha \in \mathbb R.
\end{cases}
\end{align*}
For each $-T_0\le t \le T_0$, the map $\alpha \rightarrow z(t,\alpha)$ is a smooth diffeomorphism.  Furthermore we have the obvious estimate
\begin{align}
|z(t,\alpha)- \alpha| \le t B, \label{808}
\end{align}
where $B>0$ is the same constant as in \eqref{805}. By integrating \eqref{807} along the characteristic line, we have
\begin{align} \label{809}
f(z(t,\alpha),t) = f(\alpha,0) \exp\left( -\int_0^t c(z(\alpha,s),s) ds \right), \qquad \forall\, \alpha \in \mathbb R, \, t\in[-T_0,T_0].
\end{align}

Now take $t_1$ sufficiently small such that (see \eqref{cond_800} for the definition of the constant $c_1$)
\begin{align*}
t_1 \le \min\{ \frac {c_1} {8B}, \; T_0 \}.
\end{align*}

By \eqref{808},  if  $|t|\le t_1$ and $|z(t,\alpha)| \le \frac {c_1} 2$, then obviously $|\alpha| \le {c_1}$. By \eqref{809}, \eqref{cond_800}, we conclude
that
\begin{align}
\psi^{\prime}(r,t) = f(r,t) \le 0, \qquad \forall\, |t|\le t_1, \; r\le \frac {c_1}2. \label{810}
\end{align}

By using a similar argument, we also obtain
\begin{align}
\psi^{\prime}(r,t) > 0, \qquad \forall\, |t| \le t_1, \; r\ge 2c_2, \label{812}
\end{align}

By \eqref{510} and the fact that $\psi_0(0)=0$, we obtain
$\psi(0,t)<0$ for all $t \in [-T_0,0)$. It follows from \eqref{810}
that
\begin{align}
\psi(r,t)<0, \qquad\forall\, -t_1\le t<0, \; r\le \frac{c_1}2.
\end{align}
Similarly using the fact that $\psi(\infty, t)=0$ and \eqref{812}, we obtain
\begin{align}
\psi(r,t)<0, \qquad \forall\, -t_1\le t<0,\; r \ge 2c_2.
\end{align}

By \eqref{cond_800} and smoothness of the local solution, there exist some $t_2>0$ sufficiently small such that
\begin{align*}
\psi(r,t)<0, \qquad \forall\;  |t|\le t_2,\, \frac {c_1}2 \le r \le 2c_2.
\end{align*}

Now obviously the claim follows if we take $t_0 =  \min\{t_1,t_2\}$.


\frenchspacing
\bibliographystyle{plain}

\end{document}